\documentclass[11pt,twoside]{amsart}
\usepackage{amsfonts}
\usepackage[russian,english]{babel}
\usepackage{amssymb}
\usepackage{amscd}
\usepackage{latexsym}
\usepackage{color}

\Russian
\catcode`\@=11
\def\omathop#1#2#3{\let\temp=#1\def\letter{#2}
  \ifcat#3_ \let\next\@@olim\else\let\next\@olim\fi\next#3}
\def\@olim{\letter\text{-}\!\temp}
\def\@@olim_#1{\mathchoice{
   \setbox0=\hbox{$\displaystyle\letter\text{-}\!\temp\!\text{-}\letter$}
   \setbox2=\hbox{$\displaystyle\temp$}
   \setbox4=\hbox{$\scriptstyle#1$}
   \dimen@=\wd4 \advance\dimen@ by -\wd2 \divide\dimen@ by2
   \def\next{\letter\text{-}\!\temp_{\hbox to 0pt{\hss$\scriptstyle#1$\hss}}
     \hskip\dimen@}
   \ifdim\wd2>\wd4 \def\next{\@olim_{#1}}\fi
   \ifdim\wd4>\wd0 \def\next{\mathop{\llap{$\letter$-}\!\temp}\limits_{#1}}\fi
   \next}
   {\@olim_{#1}}{\@olim_{#1}}{\@olim_{#1}}}

\catcode`\@=13

\theoremstyle{plain}
\newtheorem{thm}{Theorem}[section]

\newtheorem{rem}[thm]{Remark}

\theoremstyle{definition}
\newtheorem{definition}[thm]{Definition}
\numberwithin{equation}{section}
\input{tcilatex}

\begin{document}
\title{Stinespring type theorem for a finite family of maps on Hilbert $%
C^{\star }$-modules}
\author{M.~Pliev}
\address{South Mathematical Institute of the Russian Academy of Sciences\\
str. Markusa 22, Vladikavkaz, 362027 Russia}
\thanks{The author was supported by the Russian Academy of Science, grant %
\No 12-01-00623}

\begin{abstract}
The aim of this article is to extend the results of Asadi~M.B, B.V.R. Bhat,
G. Ramesh, K. Sumesh about completely positive maps on Hilbert $C^{\star }$%
-modules. We prove a Stinespring type theorem for a finite family of
completely positive maps on Hilbert $C^{\star }$-modules. We also show that
any two minimal Stinespring representations are unitarily equivalent.
\end{abstract}

\maketitle




\section{Introduction}

Stinespring representation theorem is a fundamental theorem in the theory of
completely positive maps. The study of completely positive maps is motivated
by applications of the theory of completely positive maps to quantum
information theory, where operator valued completely positive maps on $%
C^{\star}$-algebras are used as a mathematical model for quantum operations,
and quantum probability. A completely positive map $\varphi:A\rightarrow B$
of $C^{\star}$-algebras is a linear map with the property that $%
[\varphi(a_{ij})]_{i,j=1}^{n}$ is a positive element in the $C^{\star}$%
-algebra $M_{n}(B)$ of all $n\times n$ matrices with entries in $B$ for all
positive matrices $[(a_{ij})]_{i,j=1}^{n}$ in $M_{n}(A)$, $n\in\mathbb{N}$.
Stinespring \cite{S} shown that a completely positive map $%
\varphi:A\rightarrow L(H)$ is of the form $\varphi(\cdot)=S^{\star}\pi(%
\cdot)S$, where $\pi$ is a $\star$-representation of $A$ on a Hilbert space $%
K$ and $S$ is a bounded linear operator from $H$ to $K$. Theorem about the
structure of $n\times n$ matrices whose entries are linear positive maps
from $C^{\star}$-algebra $A$ to $L(H)$, known as completely $n$-positive
linear maps, were obtained by Heo \cite{H}. Hilbert $C^{\star}$-modules are
generalizations of Hilbert spaces and $C^{\star}$-algebras. In \cite{A}
Asadi had considered a version of the Stinespring theorem for completely
positive map on Hilbert $C^{\star}$-modules. Later Bhat, Ramesh and Sumesh
in \cite{B} had removed some technical conditions. Skiede in \cite{S} had
considered whole construction in a framework of the $C^{\star}$%
-correspondences. Finally Joita in \cite{J-1,J-2} had proved covariant
version of the Stinespring theorem and Radon-Nikodym theorem. In this paper,
we shall prove a version of the Stinespring theorem for a finite families of
the completely positive maps on Hilbert $C^{\star}$-modules.

\section{Preliminaries}

The goal of this section is to introduce some basic definitions and facts.
General information on $C^{\star}$-algebras, Hilbert $C^{\star}$-modules and
completely positive maps the reader can find in the books \cite{La,Ma,M,P,R}.

We denote Hilbert spaces by $H_{1},H_{2},K_{1},K_{2}$ etc and the
corresponding inner product and the induced norm by $\langle \cdot ,\cdot
\rangle $ and $||\cdot ||$ respectively. Throughout we assume that the inner
product is conjugate linear in the first variable and linear in the second
variable. The space of bounded linear operators from $H_{1}$ to $H_{2}$ is
denoted by $L(H_{1},H_{2})$ and $L(H_{1}):=L(H_{1},H_{1})$. We denote $%
C^{\star }$-algebras by $A,B$ etc. The $C^{\star }$-algebra of all $n\times
n $ matrices with entries from $A$ is denoted by $M_{n}(A)$.

\smallskip A Hilbert $C^{\star}$-module $V$ over $C^{\star}$-algebra $A$ is
a linear space which is also a right $A$-module, equipped with an $A$-valued
inner product $\langle\cdot,\cdot\rangle_{A}$ that is $V$ is $\mathbb{C}$%
-linear and $A$-linear in the second variable and conjugate linear in the
first variable such that $V$ is complete with the norm $||x||=||\langle
x,x\rangle_{A}||^{\frac{1}{2}}$. $V$ is full if the closed bilateral $\star$%
-sided ideal $\langle V,V\rangle_{A}$ of $A$ generated by $\{\langle
x,y\rangle_{A}:\,x,y\in V\}$ coincides with $A$. Remind the reader that $%
L(H_{1},H_{2})$ is a Hilbert $L(H_{1})$-module for any two Hilbert spaces $%
H_{1},H_{2}$, with the following operations: 
\begin{gather}
\text{module map}\,\,: (T,S)\mapsto TS : L(H_{1},H_{2})\times
L(H_{1})\rightarrow L(H_{1},H_{2}); \\
\text{inner product}\,\,\langle T, S\rangle\mapsto T^{\star}S :
L(H_{1},H_{2})\times L(H_{1},H_{2})\rightarrow L(H_{1}).
\end{gather}
A representation of $V$ on the Hilbert spaces $H_{1}$ and $H_{2}$ is a map $%
\Psi:V\rightarrow L(H_{1},H_{2})$ with the property that there is a $\star$%
-representation $\pi$ of $A$ on the Hilbert space $H_{1}$ such that 
\begin{equation*}
\langle \Psi(x),\Psi(y)\rangle=\pi(\langle x,y\rangle)
\end{equation*}
for all $x,y\in V$. If $V$ is full, then the $\star$-representation $\pi$
associated to $\Psi$ is unique. A representation $\Psi:V\rightarrow
L(H_{1},H_{2})$ of $V$ is nondegenerate if $[\Psi(V)(H_{1})]=H_{2}$ and $%
[\Psi(V)^{\star}(H_{2})]=H_{1}$ (here, $[Y]$ denotes the closed subspace of
a Hilbert space $Z$ generated by subset $Y\subset Z$). A map $%
\Phi:V\rightarrow L(H_{1},H_{2})$ is called \textit{completely positive} on $%
V$ if there is a linear completely positive map $\varphi:A\rightarrow
L(H_{1})$ such that 
\begin{equation*}
\langle \Phi(x),\Phi(y)\rangle=\varphi(\langle x,y\rangle)
\end{equation*}
for all $x,y\in V$.

Linear map $\varphi:A\rightarrow B $ is said to be positive if $%
\varphi(a^{\star}a)\geq 0$, for all $a\in A$. An $n\times n$ matrix $%
[\varphi_{ij}]_{i,j=1}^{n}$ of linear maps from $A$ to $B$ can be regarded
as a linear map $[\varphi]:M_{n}(A)\rightarrow M_{n}(A)$ defined by 
\begin{equation*}
[\varphi]([a_{ij}]_{i,j=1}^{n})=[\varphi_{ij}(a_{ij})]_{i,j=1}^{n}
\end{equation*}
We say that $[\varphi]$ is a completely $n$-positive linear map from $A$ to $%
B$ if $[\varphi]$ is a completely positive linear map from $M_{n}(A)$ to $%
M_{n}(B)$. If $[\varphi_{ij}]_{i,j=1}^{n}$ is a completely $n$-positive
linear map from $A$ to $B$, then $\varphi_{ii}$ is a completely positive
linear map from $A$ to $B$ for each $i\in\{1,\dots,n\}$.

\textbf{Acknowledgment.}~I am very grateful to professor Maria Joita for her
useful comments and remarks.

\section{Main result}

\label{sec2}

In this section we strengthen Bhat, Ramesh and Sumesh theorem and discuss
the minimality of the representation.

Let $V$ be a Hilbert $C^{\star}$-module over $A$ and let $H_{1},H_{2}$ be
Hilbert spaces. Let $\Phi_{i},\,i\in\{1,\dots,n\}$ be a maps $%
\Phi_{i}:V\rightarrow L(H_{1},H_{2})$.

\begin{definition}
\label{def:nar1} A $n$-tuple of maps $\Phi=(\Phi_{1},\dots,\Phi_{n})$ is
called \textit{completely positive}, if there is a completely $n$-positive
map $[\varphi]$ from $A$ to $L(H_{1})$ such that 
\begin{equation*}
[\langle \Phi_{i}(x),\Phi_{j}(y)\rangle]_{i,j=1}^{n}=[\varphi_{ij}\langle
x,y\rangle]_{i,j=1}^{n}
\end{equation*}
for every $x,y\in V$.
\end{definition}

\begin{rem}
It is obvious that every map $\Phi _{i},i\in \{1,\dots ,n\}$ is completely
positive.
\end{rem}

\begin{thm}
Let $A$ be a unital $C^{\star }$-algebra, $V$ be a Hilbert $A$-module, $%
[\varphi _{ij}]_{i,j=1}^{n}:A\rightarrow L(H_{1})$ be a $n$-completely
positive map and let $\Phi =(\Phi _{1},\dots ,\Phi _{n})$, $\Phi
_{i}:V\rightarrow L(H_{1},H_{2}),\,i\in \{1,\dots ,n\}$ be a $\left[ \varphi %
\right] $-completely positive n-tuple. Then there exists a data $(\pi
,S_{1},\dots ,S_{n},K_{1})$, $(\Psi ,W_{1},\dots ,W_{n},K_{2})$, where

\smallskip (1)~$K_{1}$ and $K_{2}$ are Hilbert spaces;

\smallskip (2)~$\Psi :V\rightarrow L(K_{1},K_{2})$ is a representation of $V$
on the Hilbert spaces $K_{1}$ and $K_{2}$, $\pi :A\rightarrow L(K_{1})$ is a
unital $\star $-homomorphism associated with $\Psi $, $S_{i}:H_{1}%
\rightarrow K_{1}$ are isometric linear operators, $W_{i}:H_{2}\rightarrow
K_{2}$ are coisometric linear operators for every $i\in \{1,\dots ,n\}$,
such that 
\begin{equation*}
\varphi _{ij}(a)=S_{i}^{\star }\pi _{A}(a)S_{j}\,\,\text{for all}\,\,\,a\in
A;\,i,j\in \{1,\dots ,n\}\,\,
\end{equation*}%
and%
\begin{equation*}
\Phi _{i}(x)=W_{i}^{\star }\Psi (x)S_{i}\text{ for all }x\in V\text{ and
every }i\in \{1,\dots ,n\}.
\end{equation*}
\end{thm}

\begin{proof}
At first we prove existence of $\pi$, $K_{1}$ and $S_{1},\dots,S_{n}$. The
more general construction like that is known in the literature (see for
example \cite{J-0}, Theorem 4.1.8), but for sake of a completeness we shall
consider it here. We denote by $(A\otimes_{\text{alg}}H_{1})^{n}$ the direct
sum of $n$ copies of the algebraic tensor product $A\otimes H_{1}$. $%
(A\otimes_{\text{alg}}H_{1})^{n}$ is a vector space with a map $%
\langle\cdot,\cdot\rangle_{0}:(A\otimes_{\text{alg}}H_{1})^{n}\times(A%
\otimes_{\text{alg}}H_{1})^{n}$ defined by formula 
\begin{equation*}
\big\langle\sum_{s=1}^{m}(a_{is}\otimes\xi_{is})_{i=1}^{n},%
\sum_{t=1}^{l}(b_{jt}\otimes\eta_{jt})_{j=1}^{n}\big\rangle_{0}=
\sum_{s,t=1}^{m,l}\sum_{i,j=1}^{n}\langle\xi_{is},\varphi_{ij}(a_{is}^{%
\star}b_{jt})\eta_{jt}\rangle
\end{equation*}
is $\mathbb{C}$-linear in its second variable. It is not difficult to check
that 
\begin{equation*}
\big(\big\langle\sum_{s=1}^{m}(a_{is}\otimes\xi_{is})_{i=1}^{n},%
\sum_{t=1}^{l}(b_{jt}\otimes\eta_{jt})_{j=1}^{n}\big\rangle_{0}\big)%
^{\star}= \big\langle\sum_{t=1}^{l}(b_{jt}\otimes\eta_{jt})_{j=1}^{n},%
\sum_{s=1}^{m}(a_{is}\otimes\xi_{is})_{i=1}^{n}\big\rangle_{0}
\end{equation*}
for all $(a_{is}\otimes\xi_{is})_{i=1}^{n},(b_{jt}\otimes%
\eta_{jt})_{j=1}^{n}\in(A\otimes_{\text{alg}}H_{1})^{n}$ and 
\begin{equation*}
\big\langle\sum_{s=1}^{m}(a_{is}\otimes\xi_{is})_{i=1}^{n},%
\sum_{s=1}^{m}(a_{is}\otimes\xi_{is})_{i=1}^{n}\big\rangle_{0}\geq 0
\end{equation*}
Let $M:=\{\zeta:(A\otimes_{\text{alg}}H_{1})^{n});\langle\zeta,\zeta%
\rangle_{0}=0\}$. $M$ is a subspace of $(A\otimes_{\text{alg}}H_{1})^{n}$.
Then by Cauchy-Schwarz Inequality, $M$ is a subspace of $(A\otimes_{\text{alg%
}}H_{1})^{n})$. Then $(A\otimes_{\text{alg}}H_{1})^{n}/M$ becomes
pre-Hilbert space with inner product defined by 
\begin{equation*}
\langle\zeta_{1}+M,\zeta_{2}+M\rangle:=\langle\zeta_{1},\zeta_{2}\rangle_{0}.
\end{equation*}
The completion of $(A\otimes_{\text{alg}}H_{1})^{n}/M$ with respect to the
topology induced by the inner product is denoted by $K_{1}$. We denote by $%
\xi_{i}$ the element in $(A\otimes H_{1})^{n}$ whose $i^{\text{th}}$
component is $1\otimes\xi$ and all other component are $0$. Now we can
define a map $S_{i}:H_{1}\rightarrow K_{1}$ by 
\begin{equation*}
S_{i}(\xi)=\xi_{i}+M
\end{equation*}
Let denote by $\xi_{a,i}$ the element in $(A\otimes_{\text{alg}}H_{1})^{n}/M$
whose $i^{\text{th}}$ component is $a\otimes\xi$ and all other component are 
$0$. Let $a\in A$. Consider the linear map $\pi(a):(A\otimes_{\text{alg}%
}H_{1})^{n}\rightarrow (A\otimes_{\text{alg}}H_{1})^{n}$ defined by 
\begin{equation*}
\pi(a)(a_{i}\otimes\xi_{i})_{i=1}^{n}=(aa_{i}\otimes\xi_{i})_{i=1}^{n}.
\end{equation*}
Linear map $\pi(a)$ can be extended by linearity and continuity to a linear
map, denoted also by $\pi(a)$, from $K_{1}$ to $K_{1}$. The fact that $%
\pi(a) $ is a representation of $A$ on $L(K_{1})$ it is showed in the same
manner as in the proof of Theorem 3.3.2 from \cite{J-0}. It is not difficult
to check that $\pi(a_{i})S_{i}\xi_{i}=\xi_{i,a}+M$. Therefore the subspace
of $K_{1}$ generated by $\pi(a_{i})S_{i}\xi_{i},\,i\in\{1,\dots,n\},\,%
\xi_{i}\in H_{1}$, $a_{i}\in A$ is exactly $(A\otimes_{\text{alg}%
}H_{1})^{n}/M$.

\smallskip Let $K_{2}:=[\{\Psi (V)S_{i}(H_{1}),i=1,...,n\}]$. Now we can
define $\Psi :V\rightarrow L(K_{1},K_{2})$ as follows: 
\begin{equation*}
\Psi (x)\Big(\sum_{i=1}^{n}\sum_{s=1}^{m}\pi (a_{is})S_{i}\xi _{is}\Big)%
:=\sum_{i=1}^{n}\sum_{s=1}^{m}\Phi _{i}(xa_{is})\xi _{is},
\end{equation*}%
where $x\in V$, $a_{is}\in A$, $\xi _{is}\in H_{1}$, $1\leq i\leq n,1\leq
s\leq m$, $m\in \mathbb{N}$. We claim that $\Psi (x)$ is well defined 
\begin{equation*}
\Big|\Big|\Psi (x)\Big(\sum_{i=1}^{n}\sum_{s=1}^{m}\pi (a_{is})S_{i}\xi _{is}%
\Big)\Big|\Big|^{2}=\Big|\Big|\sum_{i=1}^{n}\sum_{s=1}^{m}\Phi
_{i}(xa_{is})\xi _{is}\Big|\Big|^{2}=
\end{equation*}%
\begin{equation*}
=\Big\langle\sum_{s=1}^{m}\sum_{i=1}^{n}\Phi _{i}(xa_{is})\xi
_{is},\sum_{r=1}^{m}\sum_{j=1}^{n}\Phi _{j}(xa_{jr})\xi _{jr}\Big\rangle=
\end{equation*}%
\begin{equation*}
=\sum_{s,r=1}^{m}\sum_{i,j=1}^{n}\langle \xi _{is},\Phi _{i}(xa_{is})^{\star
}\Phi _{j}(xa_{jr})\xi _{jr}\rangle =
\end{equation*}%
\begin{equation*}
=\sum_{s,r=1}^{m}\sum_{i,j=1}^{n}\langle \xi _{is},\varphi _{ij}(\langle
xa_{is},xa_{jr}\rangle )\xi _{jr}\rangle =
\end{equation*}%
\begin{equation*}
=\sum_{s,r=1}^{m}\sum_{i,j=1}^{n}\langle \xi _{is},S_{i}^{\star }\pi
(a_{is}^{\star }\langle x,x\rangle a_{jr})S_{j}\xi _{jr}\rangle =
\end{equation*}%
\begin{equation*}
=\sum_{s,r=1}^{m}\sum_{i,j=1}^{n}\langle \pi (a_{is})S_{i}(\xi _{is}),\pi
(\langle x,x\rangle )\pi (a_{jr})S_{j}\xi _{jr}\rangle =
\end{equation*}%
\begin{equation*}
=\Big\langle\sum_{s=1}^{m}\sum_{i=1}^{n}\pi (a_{is})S_{i}(\xi _{is}),\pi
(\langle x,x\rangle )\Big(\sum_{r=1}^{m}\sum_{j=1}^{n}\pi (a_{jr})S_{j}\xi
_{jr}\Big)\Big\rangle\leq
\end{equation*}%
\begin{equation*}
\leq \Big|\Big|\pi (\langle x,x\rangle )\Big|\Big|\,\Big|\Big|\big(%
\sum_{r=1}^{m}\sum_{i=1}^{n}\pi (a_{i,r})S_{i}\xi _{i,r})\Big|\Big|^{2}\leq
\end{equation*}%
\begin{equation*}
\leq ||x||^{2}\Big|\Big|\big(\sum_{r=1}^{m}\sum_{i=1}^{n}\pi
(a_{i,r})S_{i}\xi _{i,r})\Big|\Big|^{2}.
\end{equation*}%
Hence $\Psi (x)$ is well defined and bounded. Hence it can be extended to
the whole of $K_{1}$. Now we prove that $\Psi $ is a representation. For
this let $x,y\in V$; $a_{is},b_{jr}\in A$; $\xi _{is},\eta _{jr}\in H_{1}$; $%
1\leq i,j\leq n$; $1\leq s\leq l$, $1\leq r\leq m$; $n,m\in \mathbb{N}$.
Then we have 
\begin{equation*}
\Big\langle\Psi (x)^{\star }\Psi (y)\Big(\sum_{r=1}^{m}\sum_{j=1}^{n}\pi
(b_{j,r})S_{j}\eta _{j,r}\Big),\sum_{s=1}^{l}\sum_{i=1}^{n}\pi
(a_{i,s})S_{i}\xi _{i,s}\Big\rangle=
\end{equation*}%
\begin{equation*}
=\Big\langle\sum_{r=1}^{m}\sum_{j=1}^{n}\Phi _{j}(yb_{jr})\eta
_{jr},\sum_{s=1}^{l}\sum_{i=1}^{n}\Phi _{i}(xa_{is})\xi _{is}\Big\rangle=
\end{equation*}%
\begin{equation*}
=\sum_{s=1}^{l}\sum_{r=1}^{m}\sum_{i,j=1}^{n}\langle \Phi
_{i}(xa_{is})^{\star }\Phi _{j}(yb_{jr})\eta _{jr},\xi _{is}\rangle =
\end{equation*}%
\begin{equation*}
=\sum_{s=1}^{l}\sum_{r=1}^{m}\sum_{i,j=1}^{n}\langle \varphi _{ij}(\langle
xa_{is},yb_{jr}\rangle )\eta _{jr},\xi _{is}\rangle =
\end{equation*}%
\begin{equation*}
=\sum_{s=1}^{l}\sum_{r=1}^{m}\sum_{i,j=1}^{n}\langle S_{i}^{\star }\pi
(a_{is}^{\star }\langle x,y\rangle a_{jr})S_{j}\eta _{jr},\xi _{is}\rangle =
\end{equation*}%
\begin{equation*}
=\Big\langle\pi (\langle x,y\rangle )\Big(\sum_{r=1}^{m}\sum_{j=1}^{n}\pi
(b_{j,r})S_{j}\eta _{j,r}\Big),\sum_{s=1}^{l}\sum_{i=1}^{n}\pi
(a_{i,s})S_{i}\xi _{i,s}\Big\rangle
\end{equation*}%
Thus $\Psi (x)^{\star }\Psi (y)=\pi (\langle x,y\rangle )$ on the dense set
and hence they are equal on $K_{1}$. Note $K_{2}\subset H_{2}$. Denote
subspace $[\Phi _{i}(V)(H_{1})]$ of the $H_{2}$ by $K_{2i}$. Let $%
W_{i}:=P_{K_{2i}},\,i\in \{1,\dots ,n\}$ the orthogonal projection from $%
H_{2}$ to $K_{2i}$. Then $W_{i}^{\star }:K_{2i}\rightarrow H_{2}$ is a
inclusion map. Hence $W_{i}W_{i}^{\star }=I_{K_{2i}}$ for every $i\in
\{1,\dots ,n\}$. Now we give a representation for $\Phi $. For every $x\in V$
and $\xi \in H_{1}$, we have 
\begin{equation*}
\Phi _{i}(x)(\xi )=W_{i}^{\star }\Psi (x)S_{i}(\xi )\,\,\text{for every}%
\,\,i\in \{1,\ldots ,n\}.
\end{equation*}
\end{proof}

\begin{definition}
Let $[\varphi ]$ and $\Phi $ be as an Theorem $3.3$. We say that a data $%
(\pi ,S_{1},\dots ,S_{n},K_{1})$, $(\Psi ,W_{1},\dots ,W_{n},K_{2})$ is a 
\textit{Stinespring representation} of $(\varphi ,\Phi )$ if conditions $%
(1)-(2)$ of Theorem $3.3$ is satisfied. Such a representation is said to be 
\textit{minimal} if

\begin{enumerate}
\item[1)] $K_{1}=[\{\pi (A)S_{i}(H_{1});i=1,...,n\}]$;

\item[2)] $K_{2}=[\{\Psi (V)S_{i}(H_{1});i=1,...,n\}]$.
\end{enumerate}
\end{definition}

\begin{thm}
Let $[\varphi ]$ and $\Phi $ be as an Theorem $3.3$. Assume that $(\pi
,S_{1},\dots ,S_{n},K_{1})$, $(\Psi ,W_{1},\dots ,W_{n},K_{2})$ and $(\pi
^{\prime },S_{1}^{\prime },\dots ,S_{n}^{\prime }\,K_{1}^{\prime })$, $(\Psi
^{\prime },W_{1}^{\prime },\dots ,W_{n}^{\prime },K_{2})$ are minimal
Stinespring representations. Then there exists unitary operators $%
U_{1}:K_{1}\rightarrow K_{1}^{\prime }$, $U_{2}:K_{2}\rightarrow
K_{2}^{\prime }$ such that

\begin{enumerate}
\item[(1)] $U_{1}S_{i}=S_{i}^{\prime },\,\forall i\in \{1,\dots ,n\}$; $%
U_{1}\pi (a)=\pi ^{\prime }(a)U_{1}$, $\forall a\in A$.

\item[(2)] $U_{2}W_{i}=W_{i}^{\prime };\,\forall i\in \{1,\dots ,n\}$; $%
U_{2}\Psi (x)=\Psi ^{\prime }(x)U_{1}$; $\forall x\in V$.
\end{enumerate}

That is a following diagram commutes, for all $a\in A$, $x\in V$, $i\in
\{1,\dots ,n\}$ 
\begin{equation*}
\begin{CD} H_{1}@>S_{i}>> K_{1}@>\pi(a)>> K_{1}@>\Psi(x)>>K_{2}@<W_{i}<<
H_{2}\\ @VV\text{Id}V @VVU_{1}V @VVU_{1}V @VVU_{2}V @VV\text{Id}V \\
H_{1}@>S'_{i}>> K'_{1}@>\pi'(a)>>K'_{1}@>\Psi'(x)>>K'_{2}@<W'_{i}<< H_{2}
\end{CD}
\end{equation*}
\end{thm}

\begin{proof}
Let us prove the existence of the unitary map $U_{1}:H_{1}\rightarrow K_{1}$%
. First define $U_{1}$ on the dense subspace --- $\text{linear span}\{\pi
(A)S_{i}(H_{1});i=1,...,n\}$. 
\begin{equation*}
U_{1}\Big(\sum_{s=1}^{m}\sum_{i=1}^{n}\pi (a_{is})S_{i}(\xi _{is})\Big):=%
\Big(\sum_{s=1}^{m}\sum_{i=1}^{n}\pi ^{\prime }(a_{is})S_{i}^{\prime }(\xi
_{is})\Big)
\end{equation*}%
where $a_{is}\in A$, $\xi _{is}\in H_{1}$, $m\in \mathbb{N}$. It is not
difficult to check that $U_{1}$ is an onto isometry. Denote the extension of 
$U_{1}$ to $K_{1}$ by $U_{1}$ itself. Then $U_{1}$ is unitary and satisfies
the condition in $(1)$. Now define $U_{2}$ on the dense subspace --- $\text{%
linear span}\{\Psi (V)S_{i}(H_{1});i=1,...,n\}$. 
\begin{equation*}
U_{2}\Big(\sum_{i=1}^{n}\sum_{s=1}^{m}\Psi (x_{is})S_{i}\xi _{is}\Big):=\Big(%
\sum_{i=1}^{n}\sum_{s=1}^{m}\Psi ^{\prime }(x_{is})S_{i}^{\prime }\xi _{is}%
\Big),
\end{equation*}%
where $x_{is}\in V$, $\xi _{is}\in H_{1}$, $m\in \mathbb{N}$. Using the fact
that $S_{i},S_{i}^{\prime }$ are isometric operators for every $i\in
\{1,\dots ,n\}$ we have 
\begin{equation*}
U_{2}\Big(\sum_{s=1}^{m}\Psi (x_{is})S_{i}\xi _{ns}\Big)=\sum_{s=1}^{m}\Psi
^{\prime }(x_{is})S_{i}^{\prime }\xi _{ns},
\end{equation*}%
and so $U_{2}(K_{2i})=K_{2i}^{\prime }$, where $K_{2i}$ $=$ $[\Psi
(V)S_{i}(H_{1})]$ and $K_{2i}^{\prime }=[\Psi ^{\prime }(V)S_{i}^{\prime
}(H_{1})]$. We can see that $U_{2}$ is well defined and can be extended to a
unitary map. For this consider 
\begin{equation*}
\Big|\Big|\Big(\sum_{i=1}^{n}\sum_{s=1}^{m}\Psi ^{\prime
}(x_{is})S_{i}^{\prime }\xi _{is}\Big)\Big|\Big|^{2}=\Big\langle%
\sum_{i=1}^{n}\sum_{s=1}^{m}\Psi ^{\prime }(x_{is})S_{i}^{\prime }\xi
_{is},\sum_{j=1}^{n}\sum_{r=1}^{m}\Psi ^{\prime }(x_{jr})S_{j}^{\prime }\xi
_{jr}\Big\rangle=
\end{equation*}%
\begin{equation*}
=\sum_{s,r=1}^{m}\sum_{i,j=1}^{n}\langle \Psi ^{\prime
}(x_{is})S_{i}^{\prime }\xi _{is},\Psi ^{\prime }(x_{jr})S_{j}^{\prime }\xi
_{jr}\rangle =
\end{equation*}%
\begin{equation*}
=\sum_{s,r=1}^{m}\sum_{i,j=1}^{n}\langle \xi _{is},S_{i}^{\prime \star }\pi
^{\prime }(\langle x_{is},x_{jr}\rangle )S_{j}^{\prime }(\xi _{jr})\rangle =
\end{equation*}%
\begin{equation*}
=\sum_{s,r=1}^{m}\sum_{i,j=1}^{n}\langle \xi _{is},\varphi _{ij}(\langle
xa_{is},xa_{jr}\rangle )(\xi _{jr})\rangle =
\end{equation*}%
\begin{equation*}
=\sum_{s,r=1}^{m}\sum_{i,j=1}^{n}\langle \xi _{is},S_{i}^{\star }\pi
(\langle x_{is},x_{jr}\rangle )S_{j}(\xi _{jr})\rangle =
\end{equation*}%
\begin{equation*}
=\sum_{s,r=1}^{m}\sum_{i,j=1}^{n}\langle \Psi (x_{is})S_{i}\xi _{is},\Psi
(x_{jr})S_{j}\xi _{jr}\rangle =
\end{equation*}%
\begin{equation*}
=\Big\langle\sum_{i=1}^{n}\sum_{s=1}^{m}\Psi (x_{is})S_{i}\xi
_{is},\sum_{j=1}^{n}\sum_{r=1}^{m}\Psi (x_{jr})S_{j}\xi _{jr}\Big\rangle=
\end{equation*}%
\begin{equation*}
\Big|\Big|\sum_{i=1}^{n}\sum_{s=1}^{m}\Psi (x_{is})S_{i}\xi _{is}\Big|\Big|%
^{2}.
\end{equation*}%
Hence $U_{2}$ is well defined and isometry, therefore $U_{2}$ can be
extended to whole of $K_{2}$. We call this extension $U_{2}$ itself.
Operator $U_{2}$ is an onto isometry. We have noticed that $(\pi
,S_{1},\dots ,S_{n},K_{1})$, $(\Psi ,W_{1},\dots ,W_{n},K_{2})$ and $(\pi
^{\prime },S_{1}^{\prime },\dots ,S_{n}^{\prime }\,K_{1}^{\prime })$, $(\Psi
^{\prime },W_{1}^{\prime },\dots ,W_{n}^{\prime },K_{2})$ are Stinespring
representations for $([\varphi ],\Phi )$. Hence for every $i\in \{1,\dots
,n\}$ we have 
\begin{equation*}
\Phi _{i}(x)=W_{i}^{\star }\Psi (x)S_{i}=W_{i}^{\prime \star }\Psi ^{\prime
}(x)S_{i}^{\prime }=
\end{equation*}%
\begin{equation*}
=W_{i}^{\prime \star }U_{2}\Psi (x)S_{i}.
\end{equation*}%
Hence 
\begin{equation*}
(W_{i}^{\star }-W_{i}^{\prime \star }U_{2})\Psi (x)S_{i}=0\Rightarrow
\end{equation*}%
\begin{equation*}
(W_{i}^{\star }-W_{i}^{\prime \star }U_{2})\Psi (x)S_{i}(\xi )=0,\,\forall
\,x\in V,\,\xi \in H_{1},\,\,i\in \{1,\dots ,n\}.
\end{equation*}%
Hence $U_{2}W_{i}=W_{i}^{\prime }$ for every $i\in \{1,\dots ,n\}$. Finally
we show that $U_{2}\Psi (x)=\Psi ^{\prime }(x)U_{1}$ on the dense subspace 
\begin{equation*}
\Big\{\sum_{s=1}^{m}\sum_{i=1}^{n}\pi (a_{is})S_{i}(\xi _{is});\,a_{is}\in
A,\,\xi _{is}\in H_{1},\,m\in \mathbb{N}\Big\}.
\end{equation*}%
We must to recall that every representation $\Psi :V\rightarrow
L(K_{1},K_{2})$ has a property $\Psi (xa)=\Psi (x)\pi (a)$ for every $x\in V$
and $a\in A$. Then using the fact that $\Psi $ and $\Psi ^{\prime }$ are
representations associated with $\pi $ and $\pi ^{\prime }$ respectively we
have 
\begin{equation*}
U_{2}\Psi (x)\Big(\sum_{s=1}^{m}\sum_{i=1}^{n}\pi (a_{is})S_{i}(\xi _{is})%
\Big)=U_{2}\Big(\sum_{s=1}^{m}\sum_{i=1}^{n}\Psi (xa_{is})S_{i}\xi _{is}\Big)%
=
\end{equation*}%
\begin{equation*}
\sum_{s=1}^{m}\sum_{i=1}^{n}\Psi ^{\prime }(xa_{is})S_{i}^{\prime }\xi
_{is}=\Psi ^{\prime }(x)\Big(\sum_{s=1}^{m}\sum_{i=1}^{n}\pi ^{\prime
}(a_{is})S_{i}^{\prime }(\xi _{is})\Big)=
\end{equation*}%
\begin{equation*}
=\Psi ^{\prime }(x)U_{1}\Big(\sum_{s=1}^{m}\sum_{i=1}^{n}\pi
(a_{is})S_{i}(\xi _{is})\Big).
\end{equation*}
\end{proof}


\begin{thebibliography}{99}
\bibitem{A} Asadi\,M.\,D. Stinspring{'}s theorem for Hilbert $C^{\star}$%
-modules /\!/ J. Operator Theory (2)(2009), V.62, P. 235--238.

\bibitem{B} Bhat\,R., Ramesh\,G., Sumesh\,K. Stinspring{'}s theorem for maps
on Hilbert $C^{\star}$-modules /\!/ Preprint, arXiv: 1001.3743.

\bibitem{H} Heo\,J. Completely multi-positive linear maps and
representations on Hilbert $C^{\star}$-modules /\!/ J. Operator Theory
41(1999), P. 3--22.

\bibitem{J-0} Joita \,M. Completely positive linear on pro-$C^{\star}$%
-algebras. ---University of Bucharest Press, 2008.---110~p.

\bibitem{J-1} Joita\,M. Covariant version of the Stinespring type theorem
for Hilbert $C^{\star}$-modules /\!/ Cent. Eur.J.Math. 9(4), 2011,
P.803--813.

\bibitem{J-2} Joita\thinspace M. Comparision of completely positive maps on
Hilbert $C^{\star }$-modules /\negthinspace / J. Math. Anal. Appl., 393
(2012) 644--650.

\bibitem{La} Lance \,E.\,C. Hilbert $C^{\star}$-modules. A toolkit for
operator algebraists.---Cambridge University Press, 1995.

\bibitem{Ma} Manuilov\,V.\,M., Troitskii\,E.\,V. Hilbert $C^{\star}$%
-modules.--- AMS, 2005.---202~p.

\bibitem{M} Murphy\,G\,J. $C^{\star}$-algebras and operator theory.---
Academic Press, 1990.---285~p.

\bibitem{P} Paulsen\,V. Completly bounded maps and operator
algebras.---Cambridge University Press, 2002.---314~p.

\bibitem{R} Raeburn\,I., Williams\,P.D. Morita equivalence and
continuous-trace $C^{\star}$-algebras.--- AMS, 1998.---327~p.

\bibitem{Sk} Skeide\,M. A factorization theorem for $\varphi$-maps /\!/
Preprint, arXiv: 1005.1396v1.

\bibitem{S} Stinspring\,F. Positive functions on $C^{\star}$-algebras /\!/
Proc. Amer. Math. Soc. (2)(1955), P. 211--216.
\end{thebibliography}
\end{document}